\documentclass[12pt, reqno]{amsart}
\usepackage{amsmath,amssymb,amsfonts, amscd,hyperref}

\newcommand{\Hmm}[1]{\leavevmode{\marginpar{\tiny%
$\hbox to 0mm{\hspace*{-0.5mm}$\leftarrow$\hss}%
\vcenter{\vrule depth 0.1mm height 0.1mm width \the\marginparwidth}%
\hbox to
0mm{\hss$\rightarrow$\hspace*{-0.5mm}}$\\\relax\raggedright #1}}}

\newtheorem{thm}{Theorem}[section]

\newtheorem{lemma}[thm]{Lemma}

\theoremstyle{definition}

\newtheorem{eg}[thm]{Example}

\newtheorem*{Remark}{Remark}

\newcommand{\R}{{\mathbb R}}

\newcommand{\al}{{\alpha}}

\newcommand{\de}{{\delta}}

\newcommand{\Om}{{\Omega}}
\newcommand{\si}{{\sigma}}

\newcommand{\lm}{{\lambda}}

\newcommand{\as}[1]{\left\langle #1\right\rangle}

\newcommand{\aV}[1]{\left\Vert #1\right\Vert}
\newcommand{\ov}[1]{\overline{ #1}}
\newcommand{\ow}[1]{\widetilde{ #1}}

\begin{document}
\title[Cheeger inequalities for unbounded graph Laplacians]
{Cheeger inequalities for unbounded graph Laplacians}

\author[F. Bauer]{Frank Bauer}
\address{Frank Bauer, Department of Mathematics, Harvard University, Cambridge, MA 02138, USA}

\author[M. Keller]{Matthias Keller}
\address{Matthias Keller, Einstein Institute of Mathematics, The Hebrew University of Jerusalem,
Jerusalem 91904, Israel}

\author[R. Wojciechowski]{Rados{\l}aw K. Wojciechowski}
\address{Rados{\l}aw K. Wojciechowski, York College of the City University of New York \\
Jamaica, NY 11451 \\ USA  }

\date{\today}



\begin{abstract} \noindent We use the concept of intrinsic metrics
to give a new definition for an isoperimetric constant of a graph.
We use this novel isoperimetric constant to prove a Cheeger-type estimate for the bottom of the spectrum which is nontrivial even if the
vertex degrees are unbounded.
\end{abstract}
\maketitle

 \section{Introduction}

In 1984 Dodziuk \cite{Do}  proved a lower bound on the spectrum of the Laplacian on
infinite graphs in terms of an isoperimetric constant. Dodziuk's bound is an analogue
of Cheeger's inequality for manifolds \cite{Ch} except for the fact that Dodziuk's estimate also contains an upper bound for the vertex degrees in the denominator. In a later paper \cite{DK} Dodziuk and Kendall expressed that it  would be desirable to have an estimate without the rather unnatural vertex degree bound. They overcame this problem in \cite{DK} by considering the  normalized Laplace operator, which is always a bounded operator, instead. However, the original problem of finding a lower bound on the spectrum of unbounded graph Laplace operators that only depends on an isoperimetric constant remained open until today.

In this paper, we solve this problem by using the concept of intrinsic metrics. More precisely, for a given weighted Laplacian, we use an intrinsic metric to redefine the boundary measure of a set. This leads to a modified definition of the isoperimetric constant for which we obtain a lower bound on the spectrum that solely depends on the constant. These estimates hold true for all weighted Laplacians (including bounded and unbounded Laplace operators). The strategy of proof is not surprising as it does not differ much from the one of \cite{Do,DK}. However, the main contribution of this note is to provide the right definition of an isoperimetric constant to solve the open problem mentioned above.

To this day, there is a vital interest in estimates of isoperimetric constants and in Cheeger-type inequalities.  For example, rather classical estimates for isoperimetric constants in terms of the vertex degree can be found in \cite{AM,DKa,M,M2}, and for relations to random walks, see
\cite{G,Woe2}.  While, for regular planar tessellations, isoperimetric constants can be computed explicitly   \cite{HJL, HS}, there are curvature estimates for arbitrary planar tessellations \cite{H,K2,KP,W}. For Cheeger inequalities on simplicial complexes, there is recent work found in \cite{PRT} and, for general weighted graphs,
see \cite{D2,KL2}.  Moreover, Cheeger estimates for the bottom of the essential spectrum and criteria for discreteness of spectrum are given in \cite{F,K, Woj1,Woj2}.  Upper bounds for the top of the (essential) spectrum and another criterium for the concentration of the essential spectrum in terms of the dual Cheeger constant are given in \cite{BHJ}. Finally, let us mention works
connecting discrete and continuous Cheeger estimates \cite{A-CPP, CGY, Ka, MMT}.

The paper is structured as follows. The set up is introduced in
the next section. The Cheeger inequalities are presented and
proven in Section~\ref{main}. Moreover, upper bounds are discussed.
A technique to incorporate
non-negative potentials into the estimate is discussed in
Section~\ref{potentials}.  Section~\ref{volume} is
dedicated to relating the exponential volume growth of a graph to
the isoperimetric constant via upper bounds while lower bounds on the isoperimetric
constant in the flavor of  curvature are presented in Section~\ref{lowerbound}. These lower bounds allow us to give examples where our estimate yields better results than all estimates known before.  


\section{The set up}
\subsection{Graphs}
Let $X$ be a countably infinite set equipped with the discrete
topology. A function $m:X\to(0,\infty)$ gives a Radon measure on
$X$ of full support via $m(A)=\sum_{x\in A}m(x)$ for $A \subseteq X$, so that $(X,m)$
becomes a discrete measure space.

A graph over $(X,m)$ is a symmetric function $b:X\times
X\to[0,\infty)$ with zero diagonal that satisfies
$$
\sum_{y\in  X}b(x, y)<\infty \quad \textup{ for  } x\in X.$$ We
can think of $x$ and $y $ as neighbors, i.e, being connected by an
edge, if $b(x,y)>0$ and we write $x\sim y$. In this case, $b(x,y)$
is the strength of the bond interaction between $x$ and $y$. For convenience we assume that there are no isolated vertices, i.e., every vertex has a neighbor. We
call $b$ \emph{locally finite} if each vertex has only finitely many neighbors.

The measure $n:X\to(0,\infty)$ given by
\begin{align*}
    n(x)=\sum_{y\in X}b(x,y) \quad \textup{ for  } x\in X.
\end{align*}
plays a distinguished role in the proof of classical Cheeger
inequalities.  In the case where $b:X\times X\to\{0,1\}$, $n(x)$ gives the
number of neighbors of a vertex $x$.

\subsection{Intrinsic metrics}

We call a pseudo metric $d$ for a graph $b$ on $(X,m)$ an \emph{intrinsic metric}
if
\begin{align*}
 \sum_{y\in X}b(x,y)d(x,y)^{2}\leq m(x)\quad \mbox{ for all } x\in X.
\end{align*}
The concept of intrinsic metrics was first studied systematically  by Sturm \cite{Stu} for strongly local regular Dirichlet forms and it was generalized to all regular Dirichlet forms by Frank/Lenz/Wingert in  \cite{FLW}. By \cite[Lemma~4.7, Theorem~7.3]{FLW} it can be seen that our definition coincides with the one of \cite{FLW}.
A possible choice  for $d$ is the path metric induced by the edge weights $w(x,y)=((m/n)(x)\wedge (m/n)(y))^{\frac{1}{2}}$, for $x\sim y,$
see e.g. \cite{Hu}. Moreover, the natural graph metric (i.e., the path metric with weights $w(x,y)=1$ for $x\sim y$) is  intrinsic  if $m\ge n$.
Intrinsic metrics for graphs were recently discovered independently in various
contexts, see e.g. \cite{FOLZ,FOLZ2,FOLZ3,GHM,HKW,Hu,
Hu2,HKMW,MU}, where certain variations of the concept also go under the name adapted metrics.

\subsection{Isoperimetric constant}\label{s:isoperimetric_constant}
In this section we use the concept of intrinsic metrics to give a refined definition of the isoperimetric constant. As it turns out, this novel isoperimetric constant is more suitable than the classical one if $n\ge m$.
Let $W\subseteq X$.
We define the boundary $\partial W$ of $W$ by
\begin{align*}
\partial W=\{(x,y)\in W\times X\setminus W\mid b(x,y)>0\}.
\end{align*}
For a given intrinsic metric $d $ we set the measure of
the boundary as
$${|\partial W|}=\sum_{(x,y)\in \partial W}b(x,y)d(x,y).$$
Note that ${|\partial W|}<\infty$ for finite $W\subseteq X$
by the Cauchy-Schwarz inequality and the assumption that
$\sum_{y}b(x,y)<\infty$. We define the \emph{isoperimetric
constant} or \emph{Cheeger constant} $\al(U)=\al_{d,m}(U)$ for
$U\subseteq X$  as
\begin{align*}
\al(U)= \inf_{W\subseteq U \mbox{{\scriptsize
finite}}}\frac{{|\partial W|}}{m(W)}.
\end{align*}
If $U=X$, we write
\begin{align*}
    \al=\al(X).
\end{align*}
For  $b:X\times X\to\{0,1\}$ and $d$ the natural graph metric  the
measure of the boundary $|\partial W|$ is number of edges leaving $W$. If additionally $m=n$, then our definition of $\al$ coincides with the classical one from \cite{DK}.

\subsection{Graph Laplacians}

Denote by $C_{c}(X)$ the space of real valued functions on $X$
with compact support. Let $\ell^{2}(X,m)$ be the space of square
summable real valued functions on $X$ with respect to the measure
$m$ which comes equipped with the scalar product
$\as{u,v}=\sum_{x\in X}u(x)v(x)m(x)$
and the norm     $\aV{u}=\aV{u}_{m}=  \as{u,u}^{\frac{1}{2}}$. Let the form $Q=Q_{b}$ with domain $D$ be given by
\begin{align*}
    Q(u)=\frac{1}{2}\sum_{x,y\in X}b(x,y)(u(x)-u(y))^{2},\quad
    D=\ov{C_{c}(X)}^{\|\cdot\|_{Q}},
\end{align*}
where
$\|\cdot\|_{Q}=(Q(\cdot)+\|\cdot\|^{2})^{\frac{1}{2}}$.
The form $Q$ defines a regular Dirichlet form on $\ell^{2}(X,m)$, see
\cite{FOTBOOK,KL}. The corresponding positive selfadjoint operator $L$  can be seen to act as
\begin{align*}
Lf(x)=\frac{1}{m(x)}\sum_{y\in X}b(x,y)(f(x)-f(y)),
\end{align*}
 (cf. \cite[Theorem~9]{KL}).
Let $\ow L$ be the extension of $L$ to $\ow F=\{f:X\to\R\mid\sum_{y\in
X}b(x,y)|f(y)|<\infty\mbox{ for all }x\in X\}$. We have $C_{c}(X)\subseteq D(L)$ if (and only if) $\ow L
C_{c}(X)\subseteq \ell^{2}(X,m)$,
see \cite[Theorem~6]{KL}. In particular, this can easily seen to be the case if the graph is locally finite or if $\inf_{x\in X}m(x)>0$. Note that
$L$ becomes a bounded operator if and only if  $m\ge n$ (cf. \cite[Theorem~9.3]{HKLW}).  In particular, if $m=n$, then $L$ is referred
to as the \emph{normalized Laplacian}.

We denote the bottom of the spectrum $\si(L)$ and the essential
spectrum $\si_{\mathrm{ess}}(L)$ of $L$ by
$$\lm_{0}(L)=\inf\si(L) \quad\mbox{and}\quad\lm_{0}^{\mathrm{ess}}(L)=\inf \si_{\mathrm{ess}}(L) .$$

\section{Cheeger inequalities}\label{main}
Let $b$ be a graph over $(X,m)$ and $d$ be an intrinsic metric.
In this section we prove the main results of the paper.

\subsection{Main results}
\begin{thm}\label{t:cheeger}
$$\lm_{0}(L)\ge\frac{\al^{2}}{2}.$$
\end{thm}

\begin{Remark} If we consider the operator that is related to the maximal form (cf. Section~\ref{volume}) which is discussed under the name Neumann Laplacian in \cite{HKLW} instead, then we can obtain a similar estimate as Theorem~\ref{t:cheeger}. One only has to redefine the isoperimetric constant by taking the infimum over all sets of finite measure. With this choice, all of our proofs work analogously.
\end{Remark}

Under the additional assumption that $L$ is bounded, we recover the classical Cheeger
inequality from \cite{F,M} which can be seen to be stronger than the one of \cite{DK} by the Taylor expansion. We say that $d\geq 1$ (respectively $d\le 1$) for neighbors if $d(x,y)\ge1$ (respectively $d(x,y)\le 1$) for all $x\sim y$,

\begin{thm}\label{t:cheeger2}If $m\ge n$ and $d\ge1$ or $d\le1$ for neighbors,
 then
$$\lm_{0}(L)\ge 1-\sqrt{1-{\al^{2}}}.$$
\end{thm}

In order to estimate  the essential spectrum let the \emph{isoperimetric constant at infinity} be given by
\begin{align*}
\al_{\infty}=\sup_{K\subseteq X\mbox{{\scriptsize finite}}}
\al(X\setminus K),
\end{align*}
which coincides with the one of  \cite{KL2} in the case of the natural graph metric and with the one of \cite{F,K} if additionally $b: X \times X \to \{0,1\}$ and $m=n$.
Note that the assumptions of the following theorem are in particular fulfilled if the graph is locally finite or if $\inf_{x\in X}m(x)>0$.

\begin{thm}\label{t:cheeger3}Assume $C_{c}(X)\subseteq D(L)$.
Then,
$$\lm_{0}^{\mathrm{ess}}(L)\ge\frac{\al_{\infty}^{2}}{2}.$$
\end{thm}

\subsection{Co-area formulae}

The key ingredients for the proof are the following well-known area and
co-area formulae. For example, they are already found in \cite{KL2}, see also \cite{Gri}.  We include a short proof for the sake of convenience. Let $\ell^{1}(X,m)=\{f:X \to \R \mid \sum_{x\in X}|f(x)|m(x)<\infty\}$.
\medskip

\begin{lemma}
Let $f\in \ell^{1}(X,m)$, $f\ge0$ and $\Om_{t}:=\{x\in X\mid f(x)>t\}$.
Then,
\begin{align*}
\frac{1}{2}    \sum_{x,y\in X} b(x,y)d(x,y)|f(x)-f(y)|=\int_{0}^{\infty}{|\partial \Om_{t}}|dt,
\end{align*}
where the value $\infty$ on both sides of the equation is allowed, and
\begin{align*}
\sum_{x\in X}f(x)m(x) =\int_{0}^{\infty}m(\Omega_{t})dt.
\end{align*}
\end{lemma}
\begin{proof}For $x,y\in X$, $x\sim y$ with $f(x) \not= f(y)$, let the interval $I_{x,y}$ be given by
$I_{x,y}:=[f(x)\wedge f(y),f(x)\vee f(y))$ and let
$|I_{x,y}|=|f(x)-f(y)|$ be the length of $I_{x,y}$. Then,
$(x,y)\in\partial \Om_{t}$ if and only if $t\in I_{x,y}$. Hence,
by Fubini's theorem,
\begin{align*}
\int_{0}^{\infty}{|\partial \Om_{t}|}dt &= \frac{1}{2}
\int_{0}^{\infty}\sum_{x,y\in X}b(x,y)d(x,y) 1_{I_{x,y}}(t)dt\\
&= \frac{1}{2}\sum_{x,y\in X}b(x,y)d(x,y)  \int_{0}^{\infty}1_{I_{x,y}}(t)dt \\
&= \frac{1}{2} \sum_{x,y\in X}b(x,y)d(x,y) |f(x)-f(y)|.
\end{align*}
Note that $x\in \Om_{t}$ if and only if $1_{(t,\infty)}(f(x))=1$.
Again, by Fubini's theorem,
\begin{align*}
\int_{0}^{\infty}m(\Om_{t})dt&=  \int_{0}^{\infty}\sum_{x\in X}m(x)1_{(t,\infty)}(f(x))dt \\
&=\sum_{x\in X}m(x) \int_{0}^{\infty}1_{(t,\infty)}(f(x))dt =\sum_{x\in X}m(x) f(x).
\end{align*}
\end{proof}

\subsection{Form estimates}

\begin{lemma}\label{l:form} For  $U\subseteq X$ and  $u\in D$ with  support in $U$ and $\|u\|_m =1$
\begin{align*}
    Q(u)\geq \frac{\al(U)^{2}}{2}.
\end{align*}
Moreover, if  $m\ge n$  and $d\ge1$ or $d\le1$ for neighbors,
then
\begin{align*}
Q(u)^{2}-2Q(u)+\al(U)^{2} \leq 0.
\end{align*}
\end{lemma}
\begin{proof}Let $u\in C_{c}(X)$.
 We calculate using the co-area formulae above with $f=u^{2}$
\begin{align*}
\al\|u\|_m^{2}&=\al\int_{0}^{\infty}m(\Om_{t})dt \leq\int_{0}^{\infty}|\partial\Om_{t}| dt \\ \displaybreak
&=\frac{1}{2}\sum_{x,y\in X}b(x,y)d(x,y)|u^{2}(x)-u^{2}(y)| \\
&\leq Q(u)^{\frac{1}{2}}\Big(\frac{1}{2}\sum_{x,y\in X}b(x,y)d(x,y)^{2}(u(x)+u(y))^{2}\Big)^{\frac{1}{2}}\\
&\leq Q(u)^{\frac{1}{2}}\Big(2\sum_{x\in X}u(x)^{2}\sum_{y\in
X}b(x,y)d(x,y)^{2}\Big)^{\frac{1}{2}}\leq{2}^{\frac{1}{2}}
Q(u)^{\frac{1}{2}}\|u\|_m,
\end{align*}
where the final estimate follows from the intrinsic metric
property. The second statement follows if we use in the above estimates
\begin{align*}
\frac{1}{2}&\sum_{x,y\in X}b(x,y)d(x,y)^{2}(u(x)+u(y))^{2}\\
&= 2\sum_{x,y\in X} b(x,y)d(x,y)^2u(x)^2-\frac{1}{2}\sum_{x,y\in X}b(x,y)d(x,y)^{2} (u(x)-u(y))^{2}
\\&\leq 2\|u\|_{m}^{2}-Q(u),
\end{align*}
where we distinguish the cases $d\ge1$ and $d\le 1$: For the first case we use that $d$ is intrinsic and that $-d(x,y)^{2}\le-1$. For the second case, we estimate $d(x,y)\le 1$ in the first line and then use $n\le m$.
The statement follows by the density of $C_{c}(X)$ in $D$.
\end{proof}

\subsection{Proof of the theorems}

\begin{proof}[Proof of Theorem~\ref{t:cheeger} and Theorem~\ref{t:cheeger2}]
By virtue of Lemma~\ref{l:form}, the statements follow  by the variational
characterization of $\lm_{0}$ via the Rayleigh Ritz quotient:
$\lm_{0}=\inf_{u\in
D,\|u\|=1} Q(u)$.
\end{proof}


\begin{proof}[Proof of Theorem~\ref{t:cheeger3}] Let  $Q_{U}$,  $U\subseteq X$,  be the restriction of $Q$ to $\ov{C_{c}(U)}^{\|\cdot\|_{Q}}$ and $L_{U}$ be the corresponding operator. Note that $Q_{U}=Q$ on $C_{c}(U)$. The assumption $C_{c}(X)\subseteq D(L)$ clearly implies $\ow LC_{c}(X)\subseteq \ell^{2}(X,m)$ which is equivalent to the fact that functions $y\mapsto b(x,y)/m(y)$ are in $\ell^{2}(X,m)$ for all $x\in X$, see \cite[Proposition~3.3]{KL}. This implies that for any finite set $K\subseteq X$ the operator $L_{X\setminus K}$ is a compact perturbation of $L$.   Thus, from Lemma~\ref{l:form} we conclude $$\lm_{0}^{\mathrm{ess}}(L)=\lm_{0}^{\mathrm{ess}}(L_{X\setminus K})\ge\lm_{0}(L_{X\setminus K})=\inf_{u\in C_{c}(X\setminus K),\|u\|=1}Q(u)\ge \frac{\al(X\setminus K)^{2}}{2}.$$ This implies the statement of Theorem~\ref{t:cheeger3}. \end{proof}

\subsection{Upper bounds for the bottom of spectrum}
In this section we show an upper bound of $\lm_{0}(L)$ by $\al$ as  in \cite{Do,DKa,DK} for uniformly discrete metric spaces.

\begin{thm} Let $d$ be an intrinsic metric such that $(X,d)$ is uniformly discrete with lower bound $\delta>0$.  Then, $\lm_{0}(L)\le \al/\delta$.
\end{thm}
\begin{proof} By assumption we have  $d\geq \delta>0$ away from the diagonal. It follows that $|\partial W|\ge \delta \sum_{(x,y)\in\partial W}b(x,y)=\delta Q(1_{W})$ for all $W\subseteq X$ finite. By the inequality $\delta\lm_{0}(L)\leq \delta Q(1_{W})/\|1_{W}\|^{2}\leq |\partial W|/m(W)$, we conclude the statement.
\end{proof}

The  example below shows that, in general, there is no upper bound by $\al$ only.

\begin{eg}
Let  $b_{0}:X\times X\to \{0,1\}$ be a $k$-regular rooted tree
with root $x_{0}\in X$ (that is, each vertex has $k$ forward neighbors). Furthermore, let $b_{1}:X\times X\to\{0,1\}$ be
such that $b_{1}(x,y)=1$ if and only if $x$ and $y$ have the same
distance to $x_0$ with respect to the natural graph distance in $b_{0}$,
and $b_{1}(x,y)=0$ otherwise. Now, let $b=b_{0}+b_{1}$,  $m\equiv1$
and let $d$ be given by the path metric with weights $w(x,y)=(n(x)\vee n(y))^{-\frac{1}{2}}$ for $x\sim y$. Then,
$\al=\al_{d,m}=0$ which can be seen by
$|\partial B_{r}|/m(B_{r})\le k^{-(r-1)/2}\to 0$ as $r\to\infty$, where $B_r$ is the set of vertices that have distance less or equal $r$ to with respect to the natural graph metric.

On the other hand, by \cite[Theorem~2]{KLW}  the heat kernel $p_{t}(x_{0},\cdot)$ of the graph $b$ equals the corresponding heat kernel on the
$k$-regular tree $b_{0}$. Hence, by a Li type theorem, see
\cite[Theorem~7.1]{HKLW} or \cite[Corollary~5.4]{KLVW}, we get $\frac{1}{t}\log p_{t}(x_{0},y)\to\lm_{0}(L)=k+1-2\sqrt{k}$ for any $y\in X$ and $t\to\infty$ (see also \cite[Corollary~6.7]{KLW}). As $\al=0$, this shows that $\al$
can yield no upper bound  without further assumptions.
\end{eg}


\section{Potentials}\label{potentials}
In this section we briefly discuss how the strategy proposed in
\cite{KL2} to incorporate potentials into the inequalities can be
applied to the new definition of the Cheeger constant. This yields a Cheeger estimate for all regular Dirichlet forms on discrete sets (cf. \cite[Theorem~7]{KL}).

Let $b$ be a graph over a discrete measure space $(X,m)$.
Furthermore, let $c:X\to[0,\infty)$ be a potential and define
\begin{align*}
    Q_{b,c}(u)=\frac{1}{2}\sum_{x,y\in X}b(x,y)(u(x)-u(y))^{2}+\sum_{x\in X}c(x)u(x)^{2}
\end{align*}
on $D(Q_{b,c})=\ov{C_{c}(X)}^{\|\cdot\|_{Q_{b,c}}}$ and let $L_{b,c}$ be the corresponding operator.

Let  $(X',b',m')$ be a copy of $(X,b,m)$.  Let $\dot X=X\cup X'$,
$\dot m:X\to(0,\infty)$ such that $\dot m\vert_{X}=m$, $\dot
m\vert_{X'}=m'$  and  let $\dot b:\dot X\times \dot
X\to[0,\infty)$ be given by $\dot b\vert_{X\times X}=b$, $\dot b\vert_{X'\times X'}=b'$, $\dot b(x,x')=c(x)=c'(x')$ for corresponding vertices $x\in X$ and $x'\in X'$ and $\dot b\equiv0$ otherwise. Then,
the restriction $Q_{\dot b, X}$ of the form $Q_{\dot b}$ on $\ell^{2}(\dot X,\dot
m)$ to $D(Q_{\dot b, X})=\ov{C_{c}(X)}^{\|\cdot\|_{Q_{\dot b}}}$
satisfies
\begin{align*}
D(Q_{b,c})=D(Q_{\dot b, X})\quad\mbox{and}\quad   Q_{b,c}=Q_{\dot b, X}.
\end{align*}

Let $d:X\times X\to[0,\infty)$ be a metric for $b$ over $(X,m)$ and assume there is a function $\de:X\to[0,\infty)$ such that
\begin{align*}
    \sum_{y\in X}b(x,y)d(x,y)^{2}+c(x)\de(x)^{2}\leq m(x)\quad \mbox{ for all } x\in X.
\end{align*}

\begin{eg}
(1) For a given intrinsic metric $d$ a possible choice for the function $\de$ is $\de(x)={((m(x)-\sum_{y\in X } b(x,y) d(x,y)^{2})/{c(x)})}^{\frac{1}{2}}$ if $c(x) > 0.$ \\
(2) Choose $d$ as the path metric induced by the edge weights
$w(x,y)=((\frac{m}{n+c})(x)\wedge (\frac{m}{n+c})(y))^{\frac{1}{2}}$ for $x\sim y$  and $\de$ as in (1). If $c>0$, then $\de>0$.
\end{eg}

We next define $\dot d$. Since we are only interested in the subgraph $X$ of $\dot X$, we do not need to bother to extend $d$ to all of $\dot X$ but only set $\dot d\vert_{X\times X}=d$ and $\dot d(y,x')=d(x,y)=\de(x)$ for $x,y\in X$ and the  corresponding vertex $x'\in X'$ of $x$.  Defining $\dot\al(X)=\dot\al_{d,m}(X)$ by
\begin{align*}
\dot\al(X)=\inf_{W\subseteq X \mbox{{\scriptsize
finite}}}\frac{|\partial W|_{\dot d}}{m(W)}.
\end{align*}
with $|\partial W|_{\dot d}=\sum_{(x,y)\in\partial W}(
b(x,y)d(x,y)+c(x)\de(x))$ implies that    $
\dot\al(X)=\al_{\dot d,\dot m}(X)$, where the right hand side is
the Cheeger constant of the subgraph $X\subseteq \dot X$ as in Section~\ref{s:isoperimetric_constant}. Hence, we get
$$\lm_{0}(L_{b,c})\ge\frac{\dot\al(X)^{2}}{2}$$
by Lemma~\ref{l:form} and the arguments from the proof  of
Theorem~\ref{t:cheeger}.


\section{Upper bounds by  volume growth}\label{volume}
In this section we relate the isoperimetric constant to the exponential volume growth of the graph. Let $b$ be a graph over a discrete measure space $(X,m)$ and let $d$ be an intrinsic metric. We
let $B_{r}(x)=\{y\in X\mid d(x,y)\leq r\}$ and define the
\emph{exponential volume growth} $\mu=\mu_{d,m}$ by
\begin{align*}
    \mu=\liminf_{r\to\infty}\inf_{x\in X}\frac{1}{r}\log \frac{m(B_{r}(x))}{m(B_{1}(x))}.
\end{align*}
Other than for the classical notions of isoperimetric constants and exponential volume growth on graphs (see  \cite{BMST, DKa, Fuj, M}), it is, geometrically, not obvious that $\al=\al_{d,m}$ and
$\mu=\mu_{d,m}$ can be related. However, given a Brooks-type theorem, the proof is rather immediate. Therefore, let the \emph{maximal form domain} be given by
\begin{align*}
    D^{\max}:=\{u\in \ell^{2}(X,m)\mid Q^{\max}(u)=\frac{1}{2}\sum_{x,y\in X}b(x,y)(u(x)-u(y))^{2}<\infty\}.
\end{align*}

\begin{thm}\label{t:volume}If $D=D^{\max}$, then $2\al\leq\mu$. In particular, this holds if one of the following assumptions is satisfied:
\begin{itemize}
  \item [(a)] The graph $b$ is locally finite and $d$ is an intrinsic path metric such that $(X,d)$ is metrically complete.
  \item [(b)] Every infinite path of vertices has infinite measure.
\end{itemize}
\end{thm}
\begin{proof} Under the assumption $D=D^{\max}$ we have $\lm_{0}(L)\leq \mu^{2}/8$ by \cite[Theorem~4.1]{HKW}. (Note that the $8$ in the denominator as opposed to the $4$ found in \cite{HKW} is explained in \cite[Remark~3]{HKW}.) Thus, the statement follows by Theorem~\ref{t:cheeger}. Note that, by \cite[Theorem~2]{HKMW} and \cite[Corollary 6.3]{HKLW},  (a) implies $D=D^{\max}$ and, by \cite[Theorem~6]{KL}, (b) implies $D=D^{\max}$.
\end{proof}

\section{Lower bounds by  curvature}\label{lowerbound}
In this section we give a lower bound on the isoperimetric constant by a quantity that is sometimes interpreted as  curvature \cite{DKa,Hu0,KLW}.
Let $b$ be a graph over  $(X,m)$ and let
$d$ be an intrinsic metric.

\subsection{The lower bound}
We fix an orientation on a subset of the edges, that is, we choose
$E_{+},E_{-}\subset X\times X$ with $E_{+}\cap E_{-}=\emptyset$
such that if $(x,y)\in E_{+}$, then $(y,x)\in E_{-}$. We define the
\emph{curvature} with respect to this orientation by $K:X\to\R$
\begin{align*}
    K(x)=\frac{1}{m(x)}\Big(\sum_{(x,y)\in E_{-}}b(x,y)d(x,y)-\sum_{(x,y)\in E_{+}}b(x,y)d(x,y)\Big)
\end{align*}
Let us give an example for a choice of $E_{\pm}$.

\begin{eg}\label{e:spheres} Let $b$ take values in $\{0,1\}$, $m$ be the vertex degree function $n$, and $d$ be the natural graph metric. For some fixed vertex $x_{0}\in X$, let $S_{r}$ be the spheres with respect to $d$ around $x_0$ and $|x|=r$ for $x\in S_{r}$. We choose $E_{\pm}$ such that outward (inward) oriented edges are in $E_+ $ $(E_-)$, i.e., $(x,y)\in E_{+}$, $(y,x)\in E_{-}$ if $x\in S_{r-1}$, $y\in S_{r}$ for some $r$ and $x\sim y$. Then $K(x)=(n_{-}(x)-n_{+}(x))/n(x)$, where
$n_{\pm}(x)=\#\{y\in S_{|x|\pm1}\mid y\sim x\}$ and $\#A$ denotes the cardinality of  $A$.
\end{eg}

The following theorem is an analogue to \cite[Lemma~1.15]{DKa} and
\cite[Proposition~3.3]{DL} which was also used in \cite{Woj1,Woj2}
to estimate the bottom of the essential spectrum.\medskip

\begin{thm}\label{t:lowerbound}{If $- K\ge k\ge0$, then $\al\ge k$.}
\end{thm}
\begin{proof}Let $W$  be a finite set and denote by $1_{W}$ the corresponding characteristic function. Furthermore, let ${\si}(x,y)=\pm d(x,y)$ for $(x,y)\in E_{\pm}$ and zero otherwise. We calculate directly
\begin{align*}
k&m(W)\leq -\sum_{x\in W}K(x)m(x)=-\sum_{x\in X}1_{W}(x)\sum_{y\in X}b(x,y)\si(x,y)\\
&=\frac{1}{2}\Big(\sum_{x\in X}1_{W}(x)\sum_{y\in X}b(x,y)\si(x,y)-\sum_{y\in X}1_{W}(y)\sum_{x\in X} b(x,y)\si(x,y)\Big)\\
&\leq \frac{1}{2}\sum_{x,y\in
X}b(x,y)d(x,y)|1_{W}(x)-1_{W}(y)|=|\partial W|,
\end{align*}
where we used $\sum_{y} b(x,y)d(x,y)<\infty$ and the antisymmetry of $\si$ in the second step. This
finishes the proof.
\end{proof}

\subsection{Example of antitrees}

In the final subsection we give an example of an antitree for which
Theorem \ref{t:lowerbound} together with Theorem \ref{t:cheeger} yields a  better estimate than the estimates known before. Recently, antitrees received some attention as they provide examples of graphs of polynomial volume growth (with respect to the natural graph metric) that are stochastically incomplete and have a spectral gap, see \cite{BK,GHM,HKW,Hu2,KLW,Woj3}.

For a given graph $b:X\times X\to\{0,1\}$ with root $x_{0}\in X$
and measure $m\equiv 1$, let $S_{r}$ be the vertices that have
natural graph distance $r$ to $x_{0}$ as above. We call a graph an
\emph{antitree} if every vertex in $S_{r}$ is connected to all
vertices in $S_{r+1}\cup S_{r-1}$ and to none in $S_{r}$.

In \cite{HKW}, it is shown that $\lm_{0}(L)=0$ whenever
$\lim_{r\to\infty}\log\#S_{r}/\log r <2$.
 It remains open by this result what happens in the case of an antitree with $\#S_{r-1}=r^{2}$. The classical Cheeger constant  $\al_{\mathrm{classical}}=\al_{1,n}$ for the normalized
Laplacian with the natural graph metric
which is given as the infimum over $\# \partial
W/n(W)$ (with $W\subseteq X$ finite) is zero. This can be
easily checked by choosing distance balls
$B_{r}=\bigcup_{j=0}^{r}S_{j}$ as test sets $W$. Hence, the estimate $\lambda_0(L)\ge (1-\sqrt{1-\al^2})\inf_{x\in X}n(x)$ with $\al = \al_{1,n}$ found in \cite{K} is trivial.

Likewise, the estimates presented in \cite{KLW} and \cite{Woj2}, which uses an unweighted  curvature, also give zero as a lower bound for the bottom of the spectrum in this case.

With Theorem \ref{t:lowerbound} we obtain a positive estimate for the Cheeger constant $\al=\al_{d,1}$ with the path metric $d(x,y)=(n(x)\vee n(y))^{-\frac{1}{2}}$, $x\sim y$ for the antitree with $\#S_{r-1}=r^{2}$. We  pick $E_{\pm}$ as in Example~\ref{e:spheres} above and obtain a positive lower bound for $-K$.  In particular, Theorem~\ref{t:lowerbound} shows that $\al>0$ and, thus, $\lm_{0}(L)>0$ by Theorem~\ref{t:cheeger} for the antitree satisfying $\#S_{r-1}=r^{2}$.
\medskip

\scriptsize{\textbf{Acknowledgements.}
F.B. thanks J\"{u}rgen Jost for introducing him to the topic and for many stimulating discussions during the last years.
The research leading to these results has received funding from the
European Research Council under the European Union's Seventh Framework
Programme (FP7/2007-2013) / ERC grant agreement n$^\circ$ 267087.

M.K. thanks Daniel Lenz for sharing
generously his knowledge about intrinsic metrics and he enjoyed
vivid discussions with Markus Seidel, Fabian Schwarzenberger and
Martin Tautenhahn.  M.K. also gratefully  acknowledges the
financial support from the German Research Foundation (DFG).

R.K. thanks J{\'o}zef Dodziuk for introducing him to Cheeger constants
years ago and for numerous inspiring discussions since then.
His research is partially financed by PSC-CUNY research grants, years 42 and 43.
}

\end{document}